\documentclass[11pt,a4paper]{amsart}
\usepackage[utf8]{inputenc}
\usepackage{amsmath}
\usepackage{amsfonts}
\usepackage{amssymb}
\usepackage{amsthm}
\usepackage{setspace}
\usepackage{hyperref}
\usepackage{graphicx}
\usepackage{stmaryrd}
\usepackage[shortlabels]{enumitem}
\usepackage{mathtools}
\usepackage{subcaption}
\usepackage[T1]{fontenc}
\usepackage{xcolor}
\usepackage{float}
\usepackage{tikz-cd}
\usepackage{tikz}
\usepackage[noadjust]{cite}
\usepackage{ragged2e}
\usepackage[margin=1.4in]{geometry}
\graphicspath{ {./images/} }
\newtheorem{theorem}{Theorem}[section]
\newtheorem{proposition}[theorem]{Proposition}
\newtheorem{lemma}[theorem]{Lemma}

\theoremstyle{definition}

\newtheorem{remark}[theorem]{Remark}

\newcommand{\Helly}{\mathrm{Helly}}

\newcommand{\ZZ}{\mathbb{Z}}


\title[Helly complexes are Hellyfications of their boundaries] {Helly complexes are Hellyfications of their boundaries}
\author[M.~Blufstein]{Martín Blufstein}
\address{Departamento de Matemática-IMAS, FCEyN, Universidad de Buenos Aires, Buenos Aires, Argentina}
\email{mblufstein@dm.uba.ar}

\date{}

\begin{document}

\begin{abstract}
\centering \justifying We prove that every finite Helly complex is isomorphic, at the level of its $1$-skeleton, to the Hellyfication (equivalently, the discrete injective hull) of its combinatorial boundary equipped with the metric induced from the complex. In particular, the boundary-rigidity phenomenon for Helly complexes proved by Blufstein-Chalopin-Chepoi admits a canonical injective-hull interpretation, independent of reconstruction procedures based on dismantling.
\end{abstract}

\maketitle

\section{Introduction}

A recurring theme in metric graph theory is that an object may be determined by a distinguished subset on which only the ambient metric is visible.
For finite simplicial complexes, a natural distinguished subset is the combinatorial boundary.
This question can be traced back to a classical result in phylogeny by Buneman~\cite{Buneman1974} and Zareckii~\cite{Zareckii1965}, stating that trees can be reconstructed from the pairwise distances between their leaves; and to a conjecture by Michel~\cite{Michel1981} asking whether all simple compact Riemannian manifolds can be determined up to isometry by its boundary distance function (see \cite{BessonCourtoisGallot1995,BuragoIvanov2010,PestovUhlmann2005} for positive results).
In the graph theoretic setting, this point of view was developed in the discrete boundary-rigidity program initiated by Haslegrave based on a question by Benjamini~\cite{Haslegrave2023} and subsequently pursued for CAT(0) cube complexes and other nonpositively curved complexes (see \cite{BCC2025boundary,ChalopinChepoi2024,Haslegrave2023,HaslegraveScottTamitegamaTan2025}).
In the Helly setting, the author, Chalopin and Chepoi proved that finite Helly complexes are determined, up to isomorphism, by the distances between boundary vertices measured in the $1$-skeleton; moreover, they obtained a polynomial-time reconstruction procedure \cite{BCC2025boundary}.

Helly graphs are the connected graphs where every family of pairwise intersecting balls has nonempty intersection.
Let $(B,d)$ be a finite metric space with integer-valued distances. 
Its \emph{Hellyfication} is the discrete injective hull $\Helly(B)$, which can be described in terms of extremal integer-valued functions on $B$.
It is the smallest Helly graph containing $B$ isometrically in the sense that every isometric embedding of $B$ into a finite Helly graph extends to an isometric embedding of $\Helly(B)$.
The Hellyfication can be thought of as the discrete counterpart of injective hulls of arbitrary metric spaces (see \cite{isbell1964six,dress1984trees}).

Our main theorem is the following.

\begin{theorem}\label{t:main}
Let $X$ be a connected and finite Helly complex, and let $\partial X$ be its combinatorial boundary. Equip the vertex set of $\partial X$ with the metric induced from $X$.
Then there is an isomorphism
\[
G(X)\cong \Helly(\partial X).
\]
More precisely, the isomorphism is induced by the boundary-distance map
\[
\Phi : V(X)\to \Helly(\partial X), \qquad \Phi(x)(\cdot)=d_X(x,\cdot).
\]
Under this identification, the restriction of $\Phi$ to $V(\partial X)$ is the canonical embedding of $\partial X$ into its Hellyfication.
\end{theorem}

An immediate consequence of this theorem is that any vertex of $X$ can be uniquely distinguished by its distances to the boundary vertices.
The key technical ingredient is the following boundary control property (see Proposition~\ref{p:boundary_control}):
given vertices $z,x\in X$, there is a boundary vertex $w$ such that $x$ lies on a shortest $(z,w)$-path.
Once this is established, we show that the boundary-distance maps of vertices satisfy the extremality condition defining the Hellyfication.

There is an important distinction between the present result and the boundary-rigidity theorem of \cite{BCC2025boundary}.
Boundary rigidity asserts that the boundary metric determines a Helly complex.
Theorem~\ref{t:main} identifies the reconstruction object itself: the reconstruction is the Hellyfication of the boundary metric.
Thus the result gives a universal-property interpretation of boundary rigidity, rather than a reconstruction algorithm.
The proof uses a different strategy: instead of reconstructing $X$ inductively, we construct an explicit isometric embedding of $X$ into the injective hull of its boundary.

The rest of the paper is organized as follows.
In Section~\ref{s:preliminaries} we fix graph and simplicial complex conventions, define the combinatorial boundary, and present Helly graphs together with some well-known properties. 
In Section~\ref{s:hellyfication} we review the Hellyfication of a finite integer-valued metric space in terms of extremal integer metric forms.
In Section~\ref{s:proofs} we establish the main technical result and use it to prove Theorem~\ref{t:main}.

\section{Preliminaries}\label{s:preliminaries}

\subsection{Graphs and simplicial complexes}
We begin with some standard definitions of graphs and simplicial complexes, and fix some notation that will be used across the article.
All simplicial complexes and graphs considered in this article are finite.

A \emph{simplicial complex} $X$ on a finite set $V$ is a subset of the power set of $V$ (without the empty subset), whose elements are called \emph{simplices}, such that any nonempty subset of a simplex is also a simplex.
The dimension of a simplex $\sigma$ is $|\sigma|-1$ and the dimension of $X$ is the largest dimension of a simplex of $X$.
If $\sigma$ is a simplex of dimension $k$, then all simplices of $X$ contained in $\sigma$ and having dimension $k-1$ are called the \emph{facets} of $\sigma$.
Simplices of $X$ are called \emph{maximal} if they are maximal by inclusion, and \emph{non-maximal} otherwise.
The \emph{dimension} of a simplicial complex is the maximal dimension of its simplices.
A \emph{graph} is a simplicial complex whose dimension is at most $1$.

For a simplicial complex $X$, the \emph{$k$-skeleton}  $X^{(k)}$  of $X$ consists of all simplices of dimension at most $k$.
We use the notations $V(X) = X^{(0)}$ for the set of vertices
(0-simplices) of $X$ and $G(X):=X^{(1)}$ for the 1-skeleton of $X$.
The \emph{clique complex} of a graph $G$, denoted $K(G)$, is the simplicial complex whose simplices are the vertex sets of the cliques of $G$.
A simplicial complex $X$ is a \emph{flag simplicial complex} if $X=K(G(X))$.
Note that flag complexes are uniquely determined by their 1-skeleton.

When two vertices $u,v$ form a simplex in a simplicial complex $X$, we write $u \sim v$.
All maps between simplicial complexes $X$ and $Y$ (or between graphs) in this article are \emph{simplicial maps}.
That is, maps $f:V(X) \to V(Y)$ such that if $v \sim w$ in $X$ then $f(v) = f(w)$ or $f(v) \sim f(w)$ in $Y$.

Every connected simplicial complex $X$ (or graph $G$) defines a metric space, which we indistinguishably call $X$ (or $G$) as well.
It is a finite metric space whose underlying set is $V(X)$, and the distance $d_X(u,v)$ between two vertices $u$ and $v$ of $X$ is the least amount of edges in a $(u,v)$-path in $G(X)$.

Given two vertices $u,v \in V(X)$, the \emph{interval} $I_X(u,v)$ between $u$ and $v$ consists of all vertices on shortest $(u,v)$-paths, that is, of all vertices (metrically) between $u$ and $v$: 
\[
I_X(u,v)=\{ x\in V(X): d_X(u,x)+d_X(x,v)=d_X(u,v)\}.
\]
For a vertex $v$ of $X$, the \emph{neighbourhood} of $v$ is
\[
N[v]=\{ x\in V(X): d_X(v,x)\leq 1\}.
\]

For a  simplicial complex $X$, the \emph{combinatorial boundary}
$\partial X$ is the downward closure of all non-maximal simplices of $X$ such that each of them is a facet of a unique simplex of $X$.
In this article we consider $\partial X$ as a metric space whose underlying set is $V(\partial X)$, and with the induced distance from $X$.

The one-vertex complex is a harmless degenerate case.
We adopt the convention that if $X$ consists of a single vertex, then $\partial X=X$.
For every other finite Helly complex, the boundary is nonempty, as noted in Remark~\ref{r:single_vertex} below.

A vertex $u$ of a simplicial complex $X$ is \emph{dominated} by another vertex $v$ if $N[v]$ contains $N[u]$.

\begin{lemma}[{\cite[Lemma 2.5]{BCC2025boundary}}]\label{l:dominated_neighborhood}
    Let $X$ be a simplicial complex, and $u$ and $v$ vertices of $X$ such that $v$ dominates $u$.
    Then $N[u]\setminus \{ v\}\subseteq \partial X$. 
\end{lemma}

\subsection{Helly graphs} 

\emph{Helly graphs} are the connected graphs in which the balls satisfy the Helly property, i.e., any collection of pairwise intersecting balls has a nonempty intersection.
We call the clique complexes of Helly graphs \emph{Helly complexes}.
They have recently attracted a lot of attention in Geometric Group Theory due to the fact that groups acting ``nicely'' on them satisfy strong geometric, algebraic and algorithmic properties (see \cite{ChChHiOs,ChalopinChepoiGenevoisHiraiOsajda2025,huang2021helly}).
Helly complexes have been characterized in \cite{ChChHiOs} as the flag, simply connected clique-Helly simplicial complexes (clique-Hellyness means that the set of maximal simplices satisfies the Helly property).

We continue with a property of Helly graphs that we require in this article.

\begin{theorem}[{\cite[Theorem 1(iv)]{BaPe}}]\label{t:helly_distance} Let $G$ be a Helly graph (or complex), and let $z,u\in V(G)$ with $k=d_G(z,u)\geq 2$. 
Then there exists a vertex $v$ with $d_G(z,v)=k-1$ such that $v$ is adjacent to $u$ and to every neighbour $w$ of $u$ satisfying $d_G(z,w)\le k$. 
\end{theorem}

\begin{remark}\label{r:single_vertex}
    Combining Lemma~\ref{l:dominated_neighborhood} and Theorem~\ref{t:helly_distance} by taking vertices at maximum distance, one can see that any Helly complex with at least two vertices has a nonempty boundary, moreover its boundary has at least two vertices.
    This is why, for Theorem~\ref{t:main} to hold, we require that the boundary of a simplicial complex with a single vertex is itself.
\end{remark}

\section{Hellyfication}\label{s:hellyfication}

Let $(B,d)$ be a finite metric space whose distance function is integer-valued.
An \emph{integer metric form} on $B$ is a function $f:B\to \ZZ$
satisfying
\[
f(a)+f(b)\geq d(a,b) \qquad (\forall \, a,b\in B).
\]

Let $\Delta^0(B)$ denote the set of such functions.
Given $f,g \in \Delta^0(B)$, we say $f \leq g$ if $f(x) \leq g(x)$ for each $x \in B$.
An element $f\in\Delta^0(B)$ is called \emph{extremal} if it is minimal with respect to this order: whenever $g\in\Delta^0(B)$ and $g\leq f$, then $g=f$.
We denote by $E^0(B)$ the set of extremal integer metric forms and give it the supremum metric $d_\infty$.
It can be made into a graph by putting $f \sim g$ if $d_\infty(f,g) = 1$.
The following theorem is a discrete analogue of Isbell's theorem on injective hulls and can be found in different formulations in \cite{ChalopinChepoiGenevoisHiraiOsajda2025, JPM, Pes87, Pes88}.

\begin{theorem}\label{t:hellyfication}
Let $(B,d)$ be a finite integer-valued metric space.
Then $E^0(B)$ is a Helly graph and the canonical map
\[
\iota_B : B \to E^0(B), \qquad b\to d(b,\cdot),
\]
is an isometric embedding.
It is the smallest Helly graph into which $B$ embeds isometrically.
Moreover, if $B$ embeds isometrically into a Helly graph $H$, then the embedding can be extended to an isometric embedding
\[
E^0(B) \hookrightarrow H.
\]
\end{theorem}

The graph $E^0(B)$ is usually called the \emph{Hellyfication} of $B$ and denoted $\Helly(B)$.
We record the following standard characterization of extremal integer metric forms (see for example \cite[Section 3.2]{ChalopinChepoiGenevoisHiraiOsajda2025}), and include a proof for completeness.

\begin{lemma}\label{l:extremal_characterization}
Let $f\in\Delta^0(B)$. Then $f$ is extremal if and only if, for every $a\in B$ there exists $b\in B$ such that $f(a)+f(b)=d(a,b)$.
\end{lemma}

\begin{proof}
Suppose first that $f$ is extremal, and assume there exists $a \in B$ such that for all $b\in B$ we have that $f(a)+f(b) > d(a,b)$.
We define $T_af:B\to \ZZ$ as
\[
T_af(b) = \begin{cases} 
 f(b) & \text{if } a \neq b \\ 
 \max\{\{0\}\cup\{d(a,c)-f(c) \mid c\in B\}\} & \text{if } a = b 
\end{cases}
\]
It is clear that $T_af$ is an integer metric form.
Since $f$ is an integer metric form, then $T_af\leq f$.
Also, since $f(a) > d(a,b) - f(b)$ for every $b\in B$ and $B$ is finite we have that $f(a) > T_af$, a contradiction with $f$ being extremal.

Conversely, let $g\in\Delta^0(B)$ such that $g\leq f$.
Fix $a\in B$ and choose $b\in B$ with
\[
f(a)+f(b)=d(a,b).
\]
Therefore, since $g$ is a metric form
\[
f(a)+f(b) = d(a,b) \leq g(a)+g(b) \leq g(a)+f(b),
\]
and hence $f(a)\leq g(a)$.
So, $g(a) = f(a)$.
Since $a$ was arbitrary, $g=f$.
\end{proof}

\section{Boundary control and proof of the main theorem}\label{s:proofs}

We start with the key technical result needed for the proof of Theorem~\ref{t:main}.
The strength of this proposition is that it allows to distinguish any two distinct vertices by their distance to a single boundary vertex.

\begin{proposition}\label{p:boundary_control}
Let $X$ be a finite Helly complex, and let $z\in V(X)$.
For every vertex $x\in V(X)$ there exists a vertex $w \in V(\partial X)$ such that $x\in I_X(z,w)$.
\end{proposition}

\begin{proof}
If $x=z$, then any $w \in V(\partial X)$ works, so assume that $x\neq z$.

We construct a path starting at $x$ along which the distance from $z$
strictly increases until we reach the boundary of $X$.
Suppose that $x_i$ has already been chosen and that
$x_i \notin V(\partial X)$. Put $k=d_X(z,x_i)$.
We claim that $x_i$ has a neighbour $x_{i+1}$ satisfying
\[
d_X(z,x_{i+1})=k+1.
\]

Suppose not. Then every neighbour $y$ of $x_i$ satisfies
\[
d_X(z,y) \leq k=d_X(z,x_i).
\]
If $k=1$, then $z$ is adjacent to every neighbour of $x_i$, so it dominates $x_i$.
Then by Lemma~\ref{l:dominated_neighborhood} we have that $x_i\in V(\partial X)$, a
contradiction.

Assume now that $k\geq 2$, and apply Theorem~\ref{t:helly_distance} to $z$ and
$x_i$. We obtain a vertex $v\sim x_i$ with
\[
d_X(z,v)=k-1,
\]
and such that $v$ is adjacent to every neighbour $y$ of $x_i$ with
$d_X(z,y)\le k$.
By our assumption, every neighbour of $x_i$ satisfies this
inequality.
Thus $N[x_i]\subseteq N[v]$,
so $v$ dominates $x_i$.
Again Lemma~\ref{l:dominated_neighborhood}
contradicts $x_i \notin V(\partial X)$.

Starting with $x_0=x$, we can therefore construct a sequence
\[
 x_0,x_1,x_2,\ldots
\]
such that consecutive vertices are adjacent and
\[
 d_X(z,x_{i+1})=d_X(z,x_i)+1,
\]
or $x_i \in V(\partial X)$, at which point we stop.
Since $X$ is finite, this process must terminate.
Let $w=x_m \in V(\partial X)$ be the last vertex of this sequence.
The equalities above imply
\[
 d_X(z,w)=d_X(z,x)+d_X(x,w),
\]
so $x\in I_X(z,w)$ as required.
\end{proof}

Let $X$ be a Helly complex.
Recall that we consider $\partial X$ as its vertex set with the induced metric from $X$.
For every vertex $x\in V(X)$ define
\[
f_x : \partial X \to \ZZ_{\ge0} , \qquad f_x(b)=d_X(x,b).
\]

Just as with Lemma~\ref{l:extremal_characterization}, the following fact is well-known to experts (see for example \cite{dress1984trees,isbell1964six,lang2013injective}).

\begin{lemma}\label{l:fx_extremal}
For every $x\in V(X)$, the function $f_x$ is an extremal integer metric form on $\partial X$.
\end{lemma}

\begin{proof}
The triangle inequality gives that for all $a,b\in V(\partial X)$,
\[
 f_x(a)+f_x(b)
 =d_X(x,a)+d_X(x,b)
 \geq d_X(a,b).
\]
Thus $f_x\in\Delta_0(\partial X)$.

Fix $a\in V(\partial X)$.
Proposition~\ref{p:boundary_control} gives a vertex
$b\in V(\partial X)$ such that $x\in I_X(a,b)$. Hence
\[
 d_X(a,x)+d_X(x,b)=d_X(a,b),
\]
or equivalently
\[
 f_x(a)+f_x(b)=d_X(a,b).
\]
By Lemma~\ref{l:extremal_characterization}, $f_x$ is extremal.
\end{proof}

Consequently there is a well-defined map
\[
 \Phi : V(X) \to \Helly(\partial X), \qquad \Phi(x)=f_x.
\]

\begin{proposition}\label{p:isometric_embedding}
The map $\Phi$ is an isometric embedding.
\end{proposition}

\begin{proof}
For arbitrary $x,y\in V(X)$ and $b\in V(\partial X)$, the triangle inequality gives
\[
 |f_x(b)-f_y(b)|\le d_X(x,y).
\]
Thus
\[
d_\infty(f_x,f_y) \le d_X(x,y).
\]

For the reverse inequality, take $x,y\in V(X)$.
By Proposition~\ref{p:boundary_control} there exists $b\in V(\partial X)$ such that $x\in I_X(y,b)$. Now
\[
 d_X(y,b)=d_X(y,x)+d_X(x,b),
\]
and therefore
\[
 |f_x(b)-f_y(b)| = d_X(y,b)-d(x,b) = d_X(x,y).
\]
Hence the maximum is attained and  $d_\infty(f_x,f_y) = d_X(x,y)$, so $\Phi$ is an isometric embedding.
\end{proof}

The preceding results can be summarized in a commutative diagram of isometric embeddings:
\[
\begin{tikzcd}[column sep=large,row sep=large]
\partial X \arrow[r,hook,"\iota_{\partial X}"] \arrow[d,hook] & \Helly(\partial X) \\
X \arrow[ru,hook,"\Phi"']
\end{tikzcd}
\]
where the left-hand vertical map is the inclusion of the boundary vertices into $X$.

\begin{proof}[Proof of Theorem~\ref{t:main}]
Since $X$ is a Helly graph and $\partial X$ embeds isometrically into $X$, Theorem~\ref{t:hellyfication} gives an isometric embedding
\[
\Psi : \Helly(\partial X) \hookrightarrow G(X)
\]
extending the identity on $\partial X$.

On the other hand, Proposition~\ref{p:isometric_embedding} gives an isometric embedding
\[
\Phi : G(X) \hookrightarrow \Helly(\partial X)
\]
whose restriction to $\partial X$ is the canonical embedding 
\[
\iota_{\partial X} : \partial X \to \Helly(\partial X),
\]
from Theorem~\ref{t:hellyfication}.

Since both spaces are finite, they must have the same cardinality, so the isometric embeddings are in fact isometries.
At the simplicial complex level we get
\[
X\cong K(\Helly(\partial X)),
\]
and at the graph level we get the desired isomorphism
\[
G(X)\cong \Helly(\partial X).
\]
\end{proof}

\section*{Declarations}

\subsection*{Funding} The researcher was funded by a CONICET postdoctoral grant.

\subsection*{Acknowledgements} I am thankful to Jérémie Chalopin and Victor Chepoi for introducing me to the topic and for fruitful conversations.

\bibliographystyle{amsalpha}
\bibliography{mybib}

\end{document}